\documentclass[reqno]{amsart}

\makeatletter
\@addtoreset{equation}{section}

\makeatother

\newcommand{\Rm}{\mathbb{R}}

\newcommand{\Sm}{\mathbb{S}}
\newcommand{\be}{\[}
\newcommand{\ee}{\]}
\newcommand{\ba}{\[\begin{aligned}}
\newcommand{\ea}{\end{aligned}\]}

\newcommand{\pp}{\partial}
\newcommand{\bv}[1]{\boldsymbol{\mathrm{#1}}}

\usepackage{amsmath}
\usepackage[foot]{amsaddr}
\usepackage{amsthm,amssymb,mathrsfs}
\usepackage{graphicx}
\usepackage{color}

\newtheorem{thm}{Theorem}[section]
\newtheorem{lem}[thm]{Lemma}

\newtheorem{prop}[thm]{Proposition}
\newtheorem{defn}[thm]{Definition}
\theoremstyle{remark}\newtheorem{rmk}[thm]{Remark}
\newtheorem{exa}[thm]{Example}

\title[]{Iterative inversion schemes for the Born series and the reduced inverse Born series}

\author[]{Akari Ishida$^1$}
\address{$^1$Graduate School of Science, University of Hyogo, Himeji 671-2280, Japan}
\author[]{Manabu Machida$^2$}
\address{$^2$Department of Informatics, Faculty of Engineering, Kindai University, Higashi-Hiroshima 739-2116, Japan}
\email{machida@hiro.kindai.ac.jp}

\date{\today}

\begin{document}

\begin{abstract}
Nonlinear inverse problems have complicated landscapes. Hence the calculation with naive iterative schemes (e.g., Gauss-Newton or conjugate gradients) is trapped in local minima. The (first) Born approximation can avoid this trapping but linearization is required. Nonlinear inverse problems can be solved without linearization by means of the inverse Born series. However, the computational cost of its standard recursive implementation grows exponentially when nonlinear terms are taken into account. In this work we revisit a Newton-type iterative scheme to invert the Born series and develop a fast variant. The relation between this fast scheme and the reduced inverse Born series is shown.
\end{abstract}

\maketitle

\section{Introduction}
\label{intro}

The cost function of a nonlinear inverse problem tends to have a complicated landscape and the calculation of a naive iterative method such as the Gauss-Newton or conjugate gradient \cite{Kaltenbacher-Neubauer-Scherzer08} is trapped by a local minimum \cite{Mueller-Siltanen12}. Direct methods such as the Born approximation are free from this trapping issue but these approaches require linarization of nonlinear inverse problems \cite{Isakov06}. By using the inverse Born series, nonlinear inverse problems can be solved without being trapped by local minima and without requiring linearization \cite{Moskow-Schotland19}.

After the inverse Born series was formally derived \cite{Markel-OSullivan-Schotland03,Markel-Schotland07,Panasyuk-Markel-Carney-Schotland06}, its mathematical structure was studied \cite{Moskow-Schotland08} and a recursive numerical algorithm was proposed \cite{Moskow-Schotland09}. In this way, higher-order Born approximations can be implemented with the inverse Born series and nonlinear inverse problems can be solved without linearization with the Born approach. It has been proved that the technique of the inverse Born series can be applied to different partial differential equations. The Calder\'{o}n problem was considered with the inverse Born series \cite{Arridge-Moskow-Schotland12}. Not only the diffusion equation, the inverse Born series was implemented for the radiative transport equation \cite{Machida-Schotland15}. In addition, the inverse Born series was applied to inverse problems for scalar waves \cite{Kilgore-Moskow-Schotland12} and for electromagnetic scattering \cite{Kilgore-Moskow-Schotland17}. The series was developed for discrete inverse problems \cite{Chung-Gilbert-Hoskins-Schotland17}. The technique of the inverse Born series was used to investigate the inversion of the Bremmer series \cite{Shehadeh-Malcolm-Schotland17}. The inverse Born series was extended to the Banach spaces \cite{Bardsley-Vasquez14,Lakhal18} and Sobolev $H^s$ norm \cite{Mahankali-Yang23}. A modified Born series with unconditional convergence was proposed and its inverse series was studied \cite{Abhishek-Bonnet-Moskow20}. The convergence theorem for the inverse Born series was improved after inverse operators were given in an alternative way \cite{Hoskins-Schotland22}. See Ref.~\cite{Moskow-Schotland19} for recent advances. Moreover, a reduced inverse Born series was proposed \cite{Markel-Schotland22}. The inverse-Born-series technique was extended to nonlinear partial differential equations when the nonlinear term in a differential equation was taken into account as perturbation \cite{DeFilippis-Moskow-Schotland23,DeFilippis-Moskow-Schotland24}. Recently, the inverse Born series was considered for the Helmholtz equation \cite{Cakoni-Meng-Zhou25}

The inverse Born series is not a unique way to invert the Born series. An iterative scheme to invert the Born series was proposed by Markel, O'Sullivan, and Schotland \cite{Markel-OSullivan-Schotland03}. They discussed that the iterative scheme produces, if the iteration converges, the same solution as that from the inverse Born series. The formula of this iteration is similar to Newton's method \cite{Dennis-Schnabel96}. Afterward, Bardsley and Vasquez considered the inverse Born series as the iterative method \cite{Bardsley-Vasquez14}. Otherwise, however, the scheme has never been seriously considered.

In this paper we will revisit the iterative scheme to invert the Born series. In addition, we propose a fast iterative scheme. While the computation cost grows exponentially for the recursion algorithm of the inverse Born series, the computation cost stays the same in the fast iterative method when higher-order terms are taken into account. Moreover, the proposed fast algorithm corresponds to the reduced inverse Born series.

The remainder of this paper is organized as follows. The Born series is introduced in Sec.~\ref{sec:born}. In Sec.~\ref{sec:ite}, an iterative formulation for inverting the Born series is introduced and moreover a fast iterative scheme is proposed. Convergences of the iterative schemes are proved in Sec.~\ref{sec:proofs}. In Sec.~\ref{sec:IBS}, the inverse Born series is introduced to invert the Born series. In Sec.~\ref{sec:reduce}, the reduced inverse Born series is described and the relation to the fast iterative scheme is shown. Another reduced inverse Born series is given in Sec.~\ref{sec:reduce_2}. This new reduced inverse Born series is equivalent to the proposed fast iterative scheme. Numerical results obtained by the fast iterative scheme and the inverse Born series are compared in Sec.~\ref{sec:num}. Finally, concluding remarks are given in Sec.~\ref{concl}.

\section{Born series}
\label{sec:born}

Let $\Omega$ be a bounded domain in $\Rm^d$ ($d\ge2$). We introduce $\eta\in L^{\infty}(\omega)$. Here, $\mathop{\mathrm{supp}}\eta\subseteq\omega$, where $\omega$ is a closed subset of $\Omega$ with a finite distance separating $\pp\omega$ and $\pp\Omega$. We consider the time-independent diffusion equation
\begin{equation}
\left\{\begin{aligned}
&
-D_0\Delta u+\alpha_0(1+\eta)u=f,\quad x\in\Omega,
\\
&
D_0\pp_{\nu}u+\beta u=0,\quad x\in\pp\Omega,
\end{aligned}\right.
\label{diffeq1}
\end{equation}
where $D_0,\alpha_0,\beta$ are positive constants, $f$ is a source term, and $\pp_{\nu}$ denotes the directional derivative for the outer unit normal vector $\nu$ on the boundary. Similarly we consider
\begin{equation}
\left\{\begin{aligned}
&
-D_0\Delta u_0+\alpha_0u_0=f,\quad x\in\Omega,
\\
&
D_0\pp_{\nu}u_0+\beta u_0=0,\quad x\in\pp\Omega.
\end{aligned}\right.
\label{diffeq2}
\end{equation}
We introduce the Green's function $G(x,y)$ for (\ref{diffeq2}), which obeys
\be
\left\{\begin{aligned}
&
-D_0\Delta G(x,y)+\alpha_0G(x,y)=\delta(x-y),\quad x\in\Omega,
\\
&
D_0\pp_{\nu}G(x,y)+\beta G(x,y)=0,\quad x\in\pp\Omega,
\end{aligned}\right.
\ee
where $y\in\Omega$ and $\delta(\cdot)$ is the Dirac delta function. Furthermore we let $j,n$ be integers.

By the subtraction of two diffusion equations for $u$ and $u_0$, we arrive at the following integral representaion:
\begin{equation}
u=u_0+Tu,
\label{identity}
\end{equation}
where $T$ is an operator such that
\be
(Tu)(x)=-\alpha_0\int_{\Omega}G(x,y)\eta(y)u(y)\,dy.
\ee
We assume $\|T\|<1$. Then from (\ref{identity}), we can write
\be
u=\left(I-T\right)^{-1}u_0=u_0+u_1+\cdots,
\ee
where $I$ is the identity and
\be
u_j=T^ju_0,\quad j\ge1.
\ee
Hereafter, $I$ denotes the identity on the corresponding space. We have the following recursive structure:
\be
u_j=-\alpha_0\int_{\Omega}G(x,y)\eta(y)u_{j-1}(y)\,dy,\quad j\ge1.
\ee

\begin{rmk}
If $f\in L^2(\Omega)$ and $u\in H^1(\Omega)$ is a weak solution of (\ref{diffeq1}), then $u\in H_{\rm loc}^2(\Omega)$ \cite{Evans10}. In this case, $T\colon H_{\rm loc}^2\to H_{\rm loc}^2$.
\end{rmk}

Let us set
\be
\phi(x)=u_0(x)-u(x),\quad x\in\pp\Omega.
\ee
Then the Born series can be written as
\begin{equation}
\phi=K\eta=K_1\eta+K_2\eta^{\otimes2}+K_3\eta^{\otimes3}+\cdots,
\label{Born}
\end{equation}
where $\otimes$ denotes tensor product and we used the notation:
\be
\eta^{\otimes n}=\overbrace{\eta\otimes\cdots\otimes\eta}^n,\quad n\ge1.
\ee
Multilinear operators $K_j\colon L^{\infty}(\omega^j)\to L^p(\pp\Omega)$ ($j\ge 1$, $1\le p\le\infty$) are given by
\ba
\left(K_1\eta\right)(x)&=\alpha_0\int_{\Omega}G(x,y)\eta(y)u_0(y)\,dy,
\quad x\in\pp\Omega,
\\
\left(K_j\eta^{\otimes j}\right)(x)
&=
(-1)^{j+1}\alpha_0^j\int_{\Omega}\cdots\int_{\Omega}G(x,y_1)\eta(y_1)G(y_1,y_2)\cdots
G(y_{j-1},y_j)
\\
&\times
\eta(y_j)u_0(y_j)\,dy_1\cdots dy_j,\quad x\in\pp\Omega,\quad j\ge2.
\ea
See Lemma \ref{Kjpq} for operators $K_j$. We will consider the inverse problem of determining $\eta$ from $\phi$.

\begin{rmk}
Multilinear operators $K_j\colon L^p(\omega^j)\to L^p(\pp\Omega)$ ($j\ge 1$, $2\le p\le\infty$) were considered in \cite{Moskow-Schotland08}. It is possible to introduce linear operators $K_j\colon (L^p(\omega))^{\otimes j}\to L^p(\pp\Omega)$ \cite{Bardsley-Vasquez14}.
\end{rmk}

\section{Iterative schemes}
\label{sec:ite}

We first consider the linear inverse problem of $K_1\xi=\phi$. The operator $K_1$ is bounded. Let $K_1^+$ be the Moore-Penrose pseudoinverse of $K_1$. We introduce the regularization parameter $\alpha=\alpha(\delta,\phi^{\delta})>0$ ($0<\delta<\infty$, $\phi^{\delta}\in L^p(\pp\Omega)$) such that
\be
\mathop{\mathrm{lim\,sup}}_{\delta\to0}\left\{\alpha(\delta,\phi^{\delta});\;
\phi^{\delta}\in L^p(\pp\Omega),\;\|\phi^{\delta}-\phi\|_{L^p(\pp\Omega)}\le\delta\right\}=0.
\ee
Let $\mathcal{K}_1=\mathcal{K}_{1,\alpha}\colon L^p(\pp\Omega)\to L^{\infty}(\omega)$ be a regularized pseudoinverse or regularization operator, i.e., there exists $\alpha=\alpha(\delta,\phi^{\delta})$ such that \cite{Engl-Hanke-Neubauer00}
\be
\lim_{\delta\to0}\sup\left\{\|\mathcal{K}_{1,\alpha(\delta,\phi^{\delta})}\phi^{\delta}-K_1^+\phi\|_{L^{\infty}(\omega)};\;\phi^{\delta}\in L^p(\pp\Omega),\;\|\phi^{\delta}-\phi\|_{L^p(\pp\Omega)}\le\delta\right\}=0
\ee
for all $\phi\in\mathcal{D}(K_1^+)$, where $\mathcal{D}(K_1^+)$ is the domain of $K_1^+$. We note that $\mathcal{K}_1K_1=I$ ($\alpha=0$) if $K_1$ is invertible. But $K_1$ is not invertible in general and this relation $\mathcal{K}_1K_1=I$ does not hold.

We obtain the following relation using $\mathcal{K}_1$:
\ba
\eta_{\rm proj}
&=
\mathcal{K}_1\left(\phi-K\eta+K_1\eta\right)
\\
&=
\eta_{\rm proj}-\mathcal{K}_1\left(K\eta-\phi\right),
\ea
where
\begin{equation}
\eta_{\rm proj}=\mathcal{K}_1K_1\eta.
\label{etaproj}
\end{equation}
This $\eta_{\rm proj}$ is the projection of $\eta$ onto the subspace generated by the eigenmodes with large singular values. In the inverse Born series (see Sec.~\ref{sec:IBS} below), $\eta_{\rm proj}$ is the best we can hope to reconstruct \cite{Moskow-Schotland09}.

From the above-mentioned relation, the iteration below was proposed in \cite{Markel-OSullivan-Schotland03}.
\begin{equation}
\eta^{(n+1)}=
\eta^{(n)}-\mathcal{K}_1\left(K\eta^{(n)}-\phi\right),
\quad n\ge0,
\label{iterative1}
\end{equation}
with $\eta^{(0)}\equiv0$. Since the above-mentioned Born series is obtained by the perturbation to the known absorption coefficient for the background medium, we set $\eta^{(0)}=0$. In \cite{Bardsley-Vasquez14}, nonzero initial guess $\eta^{(0)}$ was considered.

%Let $B_{\rho}:=\{\eta\in L^{\infty}(\omega);\;\|\eta\|_{L^{\infty}(\omega)}<\rho\}$ denote a ball of radius $\rho>0$ centered at the origin. The closure $\{\eta\in\ L^{\infty}(\omega);\;\|\eta\|_{L^{\infty}(\omega)}\le\rho\}$ of $B_{\rho}$ is denoted by $\overline{B_{\rho}}$. 
We write
\begin{equation}
\|\mathcal{K}_1\phi\|_{L^{\infty}(\omega)}=h,
\label{normish}
\end{equation}
where  $h\in(0,\infty)$ is a constant.% Let us set
%\be
%F[\eta]=\mathcal{K}_1\left(K\eta-\phi\right).
%\ee
%Let $DF[\eta]$ denote the Fr\'{e}chet derivative of $F$ with respect to $\eta$.

\begin{thm}
We assume that there exists a constant $b\in(1,\infty)$ such that $\|I-\mathcal{K}_1K\|\le1-1/b$. Let $(\eta^{(n)})_{n=0}^{\infty}$ be the sequence constructed by (\ref{iterative1}) with $\eta^{(0)}\equiv0$. Then the iteration in (\ref{iterative1}) admits a unique fixed-point $\eta^*$. Furthermore,
\be
\left\|\eta^{(n)}-\eta^*\right\|_{L^{\infty}(\omega)}\le 
b\left\|\eta^{(n+1)}-\eta^{(n)}\right\|_{L^{\infty}(\omega)},\quad n\ge0.
\ee
\label{mainthm1}
\end{thm}

We introduce the following fast iterative scheme:
\begin{equation}
\widetilde{\eta}^{(n+1)}=
\widetilde{\eta}^{(n)}-\mathcal{K}_1K_2\left(\widetilde{\eta}^{(n)}-\widetilde{\eta}^{(n-1)}\right)\otimes\widetilde{\eta}^{(1)},
\quad n\ge1,
\label{iterativeRed1}
\end{equation}
where $\widetilde{\eta}^{(0)}=\eta^{(0)}\equiv0$, $\widetilde{\eta}^{(1)}=\eta^{(1)}=\mathcal{K}_1\phi$.

\begin{thm}
We assume that there exists a constant $b\in(1,\infty)$ such that $\|\mathcal{K}_1K_2\|\le (1-1/b)/h$. Let $(\widetilde{\eta}^{(n)})_{n=0}^{\infty}$ be the sequence constructed by (\ref{iterativeRed1}) with $\widetilde{\eta}^{(0)}\equiv0$, $\widetilde{\eta}^{(1)}=\mathcal{K}_1\phi$. Then the iteration in (\ref{iterativeRed1}) admits a unique fixed-point $\widetilde{\eta}^*$. Furthermore,
\be
\left\|\widetilde{\eta}^{(n)}-\widetilde{\eta}^*\right\|_{L^{\infty}(\omega)}\le
b\left\|\widetilde{\eta}^{(n+1)}-\widetilde{\eta}^{(n)}\right\|_{L^{\infty}(\omega)},\quad n\ge0.
\ee
\label{mainthm2}
\end{thm}

\section{Proofs of Theorems \ref{mainthm1} and \ref{mainthm2}}
\label{sec:proofs}

\subsection{Proof of Therem \ref{mainthm1}}

For nonzero integers $k,l$,
\ba
\left\|\eta^{(k+1)}-\eta^{(l+1)}\right\|_{L^{\infty}(\omega)}
&=
\left\|\eta^{(k)}-\eta^{(l)}-\mathcal{K}_1K\left(\eta^{(k)}-\eta^{(l)}\right)\right\|_{L^{\infty}(\omega)}
\\
&\le
\left\|I-\mathcal{K}_1K\right\|\left\|\eta^{(k)}-\eta^{(l)}\right\|_{L^{\infty}(\omega)}
\\
&\le
\left(1-\frac{1}{b}\right)\left\|\eta^{(k)}-\eta^{(l)}\right\|_{L^{\infty}(\omega).}
\ea
That is, the mapping (\ref{iterative1}) from $L^{\infty}(\omega)$ to $L^{\infty}(\omega)$ is a contraction mapping and the sequence $(\eta^{(n)})_{n=0}^{\infty}$ is a Cauchy sequence. Since $L^{\infty}(\omega)$ is complete, the sequence admits a limit $\eta^*\in L^{\infty}(\omega)$. By Banach's contraction mapping theorem, there exists a unique $\eta^*$ such that
\be
\mathcal{K}_1\left(K\eta^*-\phi\right)=0.
\ee

In particular, we have (see (A.51) in \cite{Kirsch21})
\be
\left\|\eta^{(n+1)}-\eta^*\right\|_{L^{\infty}(\omega)}
\le\left(1-\frac{1}{b}\right)\left\|\eta^{(n)}-\eta^*\right\|_{L^{\infty}(\omega)}.
\ee
Since
\be
\eta^{(n)}-\eta^*=
\left(\eta^{(n)}-\eta^{(n+1)}\right)+\left(\eta^{(n+1)}-\eta^{(n+2)}\right)+\cdots
=\sum_{k=n}^{\infty}\left(\eta^{(k)}-\eta^{(k+1)}\right),
\ee
and hence
\be
\left\|\eta^{(n)}-\eta^*\right\|_{L^{\infty}(\omega)}
\le
\sum_{k=n}^{\infty}\left\|\eta^{(k)}-\eta^{(k+1)}\right\|_{L^{\infty}(\omega)}
\le
\left\|\eta^{(n+1)}-\eta^{(n)}\right\|_{L^{\infty}(\omega)}
\sum_{k=n}^{\infty}\left(1-\frac{1}{b}\right)^{k-n},
\ee
we obtain
\be
\left\|\eta^{(n)}-\eta^*\right\|_{L^{\infty}(\omega)}\le 
b\left\|\eta^{(n+1)}-\eta^{(n)}\right\|_{L^{\infty}(\omega)},\quad n\ge0.
\ee
Thus, the proof is complete.

\subsection{Proof of Theorem \ref{mainthm2}}

By using $\xi^{(n)}=\widetilde{\eta}^{(n)}-\widetilde{\eta}^{(n-1)}$ ($n\ge1$), we can express (\ref{iterativeRed1}) as
\be
\xi^{(n+1)}=-S\xi^{(n)},\quad n\ge1,
\ee
where $S\colon L^{\infty}(\omega)\to L^{\infty}(\omega)$ is defined as
\be
S\xi=\mathcal{K}_1K_2\xi\otimes\widetilde{\eta}^{(1)}
\ee
for $\xi\in L^{\infty}(\omega)$. Since $\|\mathcal{K}_1K_2\|<1/h$ by assumption, $S$ is a contraction mapping. We have
\be
\widetilde{\eta}^{(n)}=\sum_{j=1}^n\xi^{(j)}.
\ee
This equation implies that $(\widetilde{\eta}^{(n)})_{n=0}^{\infty}$ is a Cauchy sequence. Since $L^{\infty}(\omega)$ is complete, the sequence admits a limit $\widetilde{\eta}^*\in L^{\infty}(\omega)$.

By an argument similar to the proof of Theorem \ref{mainthm1}, we can show that $\|\widetilde{\eta}^{(n)}-\widetilde{\eta}^*\|_{L^{\infty}(\omega)}\le b\|\widetilde{\eta}^{(n+1)}-\widetilde{\eta}^{(n)}\|_{L^{\infty}(\omega)}$ for $n\ge0$.

\section{Inverse Born series}
\label{sec:IBS}

The Born series (\ref{Born}) expresses $\phi$ in a series form. To the contrary, we can write $\eta$ in a series form:
\begin{equation}
\eta=\eta_1+\eta_2+\eta_3+\eta_4+\cdots.
\label{invBorn}
\end{equation}
We call (\ref{invBorn}) the inverse Born series \cite{Markel-OSullivan-Schotland03,Moskow-Schotland08}. Here, the first term is given by
\be
\eta_1=\mathcal{K}_1\phi=\eta^{(1)}.
\ee
The $j$th term $\eta_j$ ($j\ge2$) in the inverse Born series can be written as
\be
\eta_j=\mathcal{K}_j\phi^{\otimes j},\quad j\ge2,
\ee
where operators $\mathcal{K}_j$ are recursively introduced as
\begin{equation}
\mathcal{K}_j=
-\left(\sum_{m=1}^{j-1}\mathcal{K}_m\sum_{i_1+\cdots+i_m=j}K_{i_1}\otimes\cdots\otimes K_{i_m}\right)\mathcal{K}_1^{\otimes j},
\quad j\ge2.
\label{arriveat}
\end{equation}
For example, we have
\begin{equation}
\mathcal{K}_2=-\mathcal{K}_1K_2\mathcal{K}_1\otimes\mathcal{K}_1,
\label{p6eq1}
\end{equation}
and
\be
\mathcal{K}_3=-\left(\mathcal{K}_2K_1\otimes K_2+\mathcal{K}_2K_2\otimes K_1+\mathcal{K}_1K_3\right)\mathcal{K}_1^{\otimes3}.
\ee
The number of ordered partitions (number of compositions) of the integer $j$ into $m$ pairs is $\begin{pmatrix}j-1\\m-1\end{pmatrix}$, i.e., the number of terms for $\mathcal{K}_j$ grows exponentially as $j$ increases. We note that
\ba
\|\eta_j\|_{L^{\infty}(\omega)}
&\le
\left\|\sum_{m=1}^{j-1}\mathcal{K}_m\sum_{i_1+\cdots+i_m=j}K_{i_1}\otimes\cdots\otimes K_{i_m}\right\|h^j
\\
&=
O(h^j).
\ea
Small $h$ is a necessary condition for the inverse Born series to converge \cite{Moskow-Schotland08}.

Since $\mathcal{K}_1K_1\colon L^{\infty}(\omega)\to L^{\infty}(\omega)$ is approximately the identity $I$, we can write
\be
\mathcal{K}_1K_1=I+\epsilon V,
\ee
where $\epsilon>0$ is small and $V$ is an operator such that $\|V\|=1$. We have
\be
\|\mathcal{K}_1K_1-I\|=\epsilon.
\ee

\begin{defn}
For two operators $\mathcal{A},\mathcal{B}\colon L^{\infty}(\omega)\to L^{\infty}(\omega)$, we write
\be
\mathcal{A}\simeq\mathcal{B}
\ee
if $\|\mathcal{A}-\mathcal{B}\|=O(\epsilon)$. Similarly, we write $\mathcal{A}\eta\simeq\mathcal{B}\eta$ ($\eta\in L^{\infty}(\omega)$) if $\mathcal{A}\simeq\mathcal{B}$.
\end{defn}

\begin{exa}
$\mathcal{K}_1K_1\simeq I$.
\end{exa}

\begin{exa}
Suppose $\eta'=\mathcal{K}_1K_1\eta$. Then we can write $\eta'\simeq\eta$.
\end{exa}

\begin{lem}
\be
\left\|K_j\eta^{\otimes j}\right\|_{L^p(\pp\Omega)}\le
C_{j,p}\|\eta\|_{L^{\infty}(\omega)}^j,
\ee
where $C_{j,p}$ is a positive constant which depends on $j,p$ ($j\ge1$, $p\ge1$).
\label{Kjpq}
\end{lem}

\begin{proof}
We have
\ba
&
\left\|K_j\eta^{\otimes j}\right\|_{L^p(\pp\Omega)}^p
\\
&=
\alpha_0^{jp}\int_{\pp\Omega}
\left|\int_{\Omega}\cdots\int_{\Omega}G(x,y_1)\eta(y_1)G(y_1,y_2)\cdots
G(y_{j-1},y_j)\eta(y_j)u_0(y_j)\,dy_1\cdots dy_j\right|^p\,dS(x)
\\
&\le
\alpha_0^{jp}\|\eta\|_{L^{\infty}(\omega)}^p\int_{\pp\Omega}\left(
\int_{\omega}\left|\int_{\omega}\cdots\int_{\omega}G(x,y_1)G(y_1,y_2)\cdots
G(y_{j-1},y_j)\eta(y_j)u_0(y_j)\,dy_2\cdots dy_j\right|\,dy_1\right)^p\,dS(x)
\\
&\le
\alpha_0^{jp}\|\eta\|_{L^{\infty}(\omega)}^{jp}\int_{\pp\Omega}\left|
\int_{\omega}\int_{\omega}\cdots\int_{\omega}G(x,y_1)G(y_1,y_2)\cdots
G(y_{j-1},y_j)u_0(y_j)\,dy_1dy_2\cdots dy_j\right|^p\,dS(x)
\\
&\le
\alpha_0^{jp}\|\eta\|_{L^{\infty}(\omega)}^{jp}
\sup_{y\in\omega}\left(\int_{\pp\Omega}|G(x,y)|^p\,dS(x)\right)\left|
\int_{\omega}\int_{\omega}\cdots\int_{\omega}G(y_1,y_2)\cdots
G(y_{j-1},y_j)u_0(y_j)\,dy_1\cdots dy_j\right|^p
\\
&\le
\alpha_0^{jp}\|\eta\|_{L^{\infty}(\omega)}^{jp}
\sup_{y\in\omega}\left(\int_{\pp\Omega}|G(x,y)|^p\,dS(x)\right)
\\
&\times
\Biggl|\left(\sup_{y_2\in\omega}\|G(\cdot,y_2)\|_{L^1(\omega)}\right)
\int_{\omega}\cdots\int_{\omega}G(y_2,y_3)\cdots
G(y_{j-1},y_j)u_0(y_j)\,dy_2\cdots dy_j\Biggr|^p
\\
&\le
\alpha_0^{jp}\|\eta\|_{L^{\infty}(\omega)}^{jp}
\sup_{y\in\omega}\left(\int_{\pp\Omega}|G(x,y)|^p\,dS(x)\right)
\Biggl|\left(\sup_{y_2\in\omega}\|G(\cdot,y_2)\|_{L^1(\omega)}\right)^{j-1}
\int_{\omega}u_0(y_j)\,dy_j\Biggr|^p
\\
&\le
C_{j,p}^p\|\eta\|_{L^{\infty}(\omega)}^{jp},
\ea
where
\be
C_{j,p}=
\alpha_0^j\sup_{y_1\in\omega}\|G(\cdot,y_1)\|_{L^p(\pp\Omega)}
\left(\sup_{y_2\in\omega}\|G(\cdot,y_2)\|_{L^1(\omega)}\right)^{j-1}
\|u_0\|_{L^1(\omega)}.
\ee
\end{proof}

Let us consider the relation between the $n$th iteration $\eta^{(n)}$ and terms $\eta_1,\dots,\eta_n$ in the inverse Born series. When $n=1$, we have $\eta^{(1)}=\eta_1$. When $n=2$, we have
\ba
\eta^{(2)}
&=
\eta^{(1)}+\mathcal{K}_1\left(\phi-K\eta_1\right)
\\
&=
\eta^{(1)}+\left(I-\mathcal{K}_1K_1\right)\eta_1+\mathcal{K}_2\phi\otimes\phi
-\mathcal{K}_1\sum_{m=3}^{\infty}K_m\eta_1^{\otimes m}
\\
&\simeq
\eta_1+\eta_2-\mathcal{K}_1\sum_{m=3}^{\infty}K_m\eta_1^{\otimes m}.
\ea
Hence $\eta^{(2)}\approx\eta_1+\eta_2$ if terms of order $O(h^3)$ are ignored. For $n=3$, we have
\ba
\eta^{(3)}
&=
\eta^{(2)}+\mathcal{K}_1\phi-\mathcal{K}_1\sum_{j=1}^{\infty}K_j\eta^{(2)\otimes j}
\\
&\simeq
\eta_1-\mathcal{K}_1\sum_{j=2}^{\infty}K_j\eta^{(2)\otimes j}
\\
&\simeq
\eta_1-\mathcal{K}_1\sum_{j=2}^{\infty}K_j\left(
\eta_1+\eta_2-\mathcal{K}_1\sum_{m=3}^{\infty}K_m\eta_1^{\otimes m}
\right)^{\otimes j}
\\
&=
\eta_1-\mathcal{K}_1K_2\left(
\eta_1+\eta_2-\mathcal{K}_1\sum_{m=3}^{\infty}K_m\eta_1^{\otimes m}
\right)^{\otimes 2}
-\mathcal{K}_1K_3\left(
\eta_1+\eta_2-\mathcal{K}_1\sum_{m=3}^{\infty}K_m\eta_1^{\otimes m}
\right)^{\otimes 3}
\\
&-
\mathcal{K}_1\sum_{j=4}^{\infty}K_j\left(
\eta_1+\eta_2-\mathcal{K}_1\sum_{m=3}^{\infty}K_m\eta_1^{\otimes m}
\right)^{\otimes j}.
\ea
Thus,
\ba
\eta^{(3)}
&\simeq
\eta_1-\mathcal{K}_1K_2\eta_1\otimes\eta_1
-\mathcal{K}_1K_2\left(\eta_1\otimes\eta_2+\eta_2\otimes\eta_1\right)
-\mathcal{K}_1K_3\eta_1^{\otimes 3}
\\
&-
\mathcal{K}_1K_2\left(\eta_2\otimes\eta_2
-\eta_1\otimes\mathcal{K}_1\sum_{m=3}^{\infty}K_m\eta_1^{\otimes m}
-\eta_2\otimes\mathcal{K}_1\sum_{m=3}^{\infty}K_m\eta_1^{\otimes m}
+\left(\mathcal{K}_1\sum_{m=3}^{\infty}K_m\eta_1^{\otimes m}\right)^{\otimes2}
\right)
\\
&
+\mathcal{K}_1K_3\eta_1^{\otimes 3}
-\mathcal{K}_1K_3\left(
\eta_1+\eta_2-\mathcal{K}_1\sum_{m=3}^{\infty}K_m\eta_1^{\otimes m}
\right)^{\otimes 3}
-\mathcal{K}_1\sum_{j=4}^{\infty}K_j\left(
\eta_1+\eta_2-\mathcal{K}_1\sum_{m=3}^{\infty}K_m\eta_1^{\otimes m}
\right)^{\otimes j}.
\ea
Hence, if we ignore terms of order $O(h^4)$, we obtain
\ba
\eta^{(3)}
&\approx
\eta_1-\mathcal{K}_1K_2\eta_1\otimes\eta_1
-\mathcal{K}_1K_2\left(\eta_1\otimes\eta_2+\eta_2\otimes\eta_1\right)
-\mathcal{K}_1K_3\eta_1^{\otimes 3}
\\
&\simeq
\eta_1+\eta_2
-\mathcal{K}_2\left(K_1\mathcal{K}_1\phi\otimes(K_2\eta_1\otimes\eta_1)\right)
-\mathcal{K}_2\left((K_2\eta_1\otimes\eta_1)\otimes K_1\mathcal{K}_1\phi\right)
-\mathcal{K}_1K_3\eta_1^{\otimes 3}
\\
&=
\eta_1+\eta_2+\eta_3.
\ea
Similarly, we observe $\eta^{(4)}\simeq\sum_{j=1}^4\eta_j+O(h^5)$ and $\eta^{(5)}\simeq\sum_{j=1}^5\eta_j+O(h^6)$.

\section{Reduced inverse Born series}
\label{sec:reduce}

In \cite{Markel-Schotland22}, the reduced inverse Born series was considered in the limited case where the matrix which corresponds to the Green's function is invertible. In this section, using Lemma \ref{dominant}, we derive the reduced inverse Born series without assuming invertible matrices but assuming that $\mathcal{K}_1K_1$ is close to the identity $I$.

We write $\widetilde{\mathcal{K}}_1=\mathcal{K}_1$, $\widetilde{\mathcal{K}}_2=\mathcal{K}_2$. For $j\ge3$, we recursively define
\be
\widetilde{\mathcal{K}}_j
=-\left(\widetilde{\mathcal{K}}_{j-1}K_2\otimes K_1^{\otimes(j-2)}\right)
\mathcal{K}_1^{\otimes j}.
\ee

\begin{lem}
The following relation holds for any $\xi_k\in L^{\infty}(\omega)$ ($1\le k\le j+n$, $j\ge1$, $n\ge1$):
\begin{equation}
K_{j+1}\xi_1\otimes\cdots\otimes \xi_j\otimes(\mathcal{K}_1K_n\xi_{j+1}\otimes\cdots\otimes \xi_{j+n})
\simeq K_{j+n}\xi_1\otimes\cdots\otimes \xi_{j+n}.
\end{equation}
\label{generalKjn}
\end{lem}

\begin{proof}
\ba
&
K_{j+1}\xi_1\otimes\cdots\otimes \xi_j\otimes(\mathcal{K}_1K_n\xi_{j+1}\otimes\cdots\otimes \xi_{j+n})
\\
&=
\int_{\omega}\cdots\int_{\omega}G(\cdot,x_1)\xi_1(x_1)G(x_1,x_2)\xi_2(x_2)G(x_2,x_3)\times\cdots\times \xi_j(x_j)G(x_j,x_{j+1})
\\
&\times
\left(\mathcal{K}_1\int_{\omega}\cdots\int_{\omega}
G(\cdot,y_1)\xi_{j+1}(y_1)G(y_1,y_2)\times\cdots\times
\xi_{j+n}(y_n)u_0(y_n)\,dy_1\dots dy_n
\right)(x_{j+1})
\\
&\times
u_0(x_{j+1})\,dx_1\dots dx_{j+1}
\\
&=
\int_{\omega}\cdots\int_{\omega}G(\cdot,x_1)\xi_1(x_1)G(x_1,x_2)\xi_2(x_2)G(x_2,x_3)\times\cdots\times \xi_j(x_j)G(x_j,x_{j+1})
\\
&\times
\Biggl(\xi_{j+1}(x_{j+1})\int_{\omega}\cdots\int_{\omega}
G(x_{j+1},y_2)\times\cdots\times \xi_{j+n}(y_n)u_0(y_n)\,dy_2\dots dy_n
\\
&+
\epsilon\left(V\xi_{j+1}(\cdot)\int_{\omega}\cdots\int_{\omega}
G(\cdot,y_2)\times\cdots\times \xi_{j+n}(y_n)u_0(y_n)\,dy_2\dots dy_n
\right)(x_{j+1})u_0(x_{j+1})\Biggr)
\\
&\times
\,dx_1\dots dx_{j+1}
\\
&=
K_{j+n}\xi_1\otimes\cdots\otimes \xi_{j+n}
\\
&+
\epsilon\int_{\omega}\cdots\int_{\omega}G(\cdot,x_1)\xi_1(x_1)\times\cdots\times \xi_j(x_j)G(x_j,x_{j+1})
\\
&\times
\left(V\xi_{j+1}(\cdot)\int_{\omega}\cdots\int_{\omega}
G(\cdot,y_2)\times\cdots\times \xi_{j+n}(y_n)u_0(y_n)\,dy_2\dots dy_n
\right)(x_{j+1})
\\
&\times
u_0(x_{j+1})\,dx_1\dots dx_{j+1}.
\ea
\end{proof}

\begin{lem}
For $n\ge2$,
\begin{equation}
\mathcal{K}_n\simeq
-\left(\mathcal{K}_{n-1}K_2\otimes K_1^{\otimes(n-2)}\right)
\mathcal{K}_1^{\otimes n}.
\label{reduction}
\end{equation}
\label{dominant}
%\label{recKsimeq}
\end{lem}

\begin{proof}
The proof can be done by mathematical induction. When $n=2$, the relation (\ref{reduction}) holds exactly (see (\ref{p6eq1})):
\be
\mathcal{K}_2=-\mathcal{K}_1K_2\mathcal{K}_1\otimes\mathcal{K}_1.
\ee
Hereafter, $n\ge3$. Suppose the relation (\ref{reduction}) holds up to $n$.

We have
\be
-\left(\mathcal{K}_1K_{n+1}\right)\mathcal{K}_1^{\otimes(n+1)}
\simeq
\left(\mathcal{K}_2K_1\otimes K_n\right)\mathcal{K}_1^{\otimes(n+1)}
\ee
because
\ba
\left(\mathcal{K}_2K_1\otimes K_n\right)\mathcal{K}_1^{\otimes(n+1)}
&=
-\left((\mathcal{K}_1K_2\mathcal{K}_1\otimes\mathcal{K}_1)K_1\otimes K_n\right)\mathcal{K}_1^{\otimes(n+1)}
\\
&\simeq
-\mathcal{K}_1K_2\left(I\otimes\mathcal{K}_1K_n\right)\mathcal{K}_1^{\otimes(n+1)}
\\
&\simeq
-\mathcal{K}_1K_{n+1}\mathcal{K}_1^{\otimes(n+1)},
\ea
where we used $\mathcal{K}_1K_1\simeq I$ and Lemma \ref{generalKjn}. Let us suppose for $1\le j\le n-1$,
\begin{equation}\begin{aligned}
&-
\left(\sum_{m=1}^{n-j}\mathcal{K}_m\sum_{i_1+\cdots+i_m=n+1}
K_{i_1}\otimes\cdots\otimes K_{i_m}\right)\mathcal{K}_1^{\otimes(n+1)}
\\
&\simeq
\left(\mathcal{K}_{n-j+1}\sum_{i_2+\cdots+i_{n-j+1}=n}K_1\otimes K_{i_2}\otimes
\cdots\otimes K_{i_{n-j+1}}\right)\mathcal{K}_1^{\otimes(n+1)}.
\end{aligned}
\label{relation1inSec6}
\end{equation}
Then we have
\begin{equation}\begin{aligned}
&-
\left(\sum_{m=1}^{n-j+1}\mathcal{K}_m\sum_{i_1+\cdots+i_m=n+1}
K_{i_1}\otimes\cdots\otimes K_{i_m}\right)
\mathcal{K}_1^{\otimes(n+1)}
\\
&=
-\left(\mathcal{K}_{n-j+1}\sum_{i_1+\cdots+i_{n-j+1}=n}
K_{i_1+1}\otimes K_{i_2}\otimes\cdots\otimes K_{i_{n-j+1}}\right)
\mathcal{K}_1^{\otimes(n+1)}
\\
&-
\left(\mathcal{K}_{n-j+1}\sum_{i_2+\cdots+i_{n-j+1}=n}
K_1\otimes K_{i_2}\otimes\cdots\otimes K_{i_{n-j+1}}\right)
\mathcal{K}_1^{\otimes(n+1)}
\\
&
-\left(\sum_{m=1}^{n-j}\mathcal{K}_m\sum_{i_1+\cdots+i_m=n+1}
K_{i_1}\otimes\cdots\otimes K_{i_m}\right)
\mathcal{K}_1^{\otimes(n+1)}
\\
&\simeq
-\left(\mathcal{K}_{n-j+1}\sum_{i_1+\cdots+i_{n-j+1}=n}
K_{i_1+1}\otimes K_{i_2}\otimes\cdots\otimes K_{i_{n-j+1}}\right)
\mathcal{K}_1^{\otimes(n+1)},
\end{aligned}
\label{CondA}
\end{equation}
where we arrived at the rightmost side by using the assumption (\ref{relation1inSec6}). For $2\le j\le n-1$, using (\ref{reduction}),
\begin{equation}\begin{aligned}
&
\left(\mathcal{K}_{n-j+2}\sum_{i_2+\cdots+i_{n-j+2}=n}K_1\otimes K_{i_2}\otimes
\cdots\otimes K_{i_{n-j+2}}\right)\mathcal{K}_1^{\otimes(n+1)}
\\
&\simeq
-\Biggl((\mathcal{K}_{n-j+1}K_2\otimes K_1^{\otimes(n-j)})\mathcal{K}_1^{\otimes(n-j+2)}
\\
&\times
\sum_{i_2+\cdots+i_{n-j+2}=n}K_1\otimes K_{i_2}\otimes
\cdots\otimes K_{i_{n-j+2}}\Biggr)
\mathcal{K}_1^{\otimes(n+1)}
\\
&\simeq
-\left(\mathcal{K}_{n-j+1}\sum_{i_2+\cdots+i_{n-j+2}=n}
K_2(I\otimes \mathcal{K}_1K_{i_2})\otimes K_{i_3}\otimes
\cdots\otimes K_{i_{n-j+2}}\right)\mathcal{K}_1^{\otimes(n+1)}
\\
&\simeq
-\left(\mathcal{K}_{n-j+1}\sum_{i_2+\cdots+i_{n-j+2}=n}K_{i_2+1}\otimes K_{i_3}\otimes
\cdots\otimes K_{i_{n-j+2}}\right)\mathcal{K}_1^{\otimes(n+1)}.
\end{aligned}
\label{CondB}
\end{equation}
Using (\ref{CondA}), (\ref{CondB}), the relation (\ref{relation1inSec6}) holds also for $j-1$. By repeating the procedure, we arrive at the relation below.
\begin{equation}
\begin{aligned}
&-
\left(\sum_{m=1}^{n-1}\mathcal{K}_m\sum_{i_1+\cdots+i_m=n+1}
K_{i_1}\otimes\cdots\otimes K_{i_m}\right)
\mathcal{K}_1^{\otimes(n+1)}
\\
&\simeq
\left(\mathcal{K}_n\sum_{i_2+\cdots+i_n=n}K_1\otimes K_{i_2}\otimes
\cdots\otimes K_{i_n}\right)\mathcal{K}_1^{\otimes(n+1)}.
\end{aligned}
\label{supplsimeq}
\end{equation}

Therefore we obtain
\ba
\mathcal{K}_{n+1}
&=
-\left(\sum_{m=1}^n\mathcal{K}_m\sum_{i_1+\cdots+i_m=n+1}
K_{i_1}\otimes\cdots\otimes K_{i_m}\right)
\mathcal{K}_1^{\otimes(n+1)}
\\
&=
-\left(\mathcal{K}_nK_2\otimes K_1^{\otimes(n-1)}\right)
\mathcal{K}_1^{\otimes(n+1)}
\\
&-
\left(\sum_{m=1}^{n-1}\mathcal{K}_m\sum_{i_1+\cdots+i_m=n+1}
K_{i_1}\otimes\cdots\otimes K_{i_m}\right)
\mathcal{K}_1^{\otimes(n+1)}
\\
&-
\left(\mathcal{K}_n\sum_{i_2+\cdots+i_n=n}K_1\otimes K_{i_2}\otimes
\cdots\otimes K_{i_n}\right)\mathcal{K}_1^{\otimes(n+1)}
\\
&\simeq
-\left(\mathcal{K}_nK_2\otimes K_1^{\otimes(n-1)}\right)
\mathcal{K}_1^{\otimes(n+1)},
\ea
where we used (\ref{supplsimeq}). The proof is complete if we recursively establish the relation (\ref{reduction}) for $n\ge3$.
\end{proof}

Lemma \ref{dominant} implies
\be
\widetilde{\mathcal{K}}_1=\mathcal{K}_1,\quad
\widetilde{\mathcal{K}}_2=\mathcal{K}_2,\quad
\widetilde{\mathcal{K}}_j\simeq \mathcal{K}_j,\quad j\ge3.
\ee
Let us introduce the reduced inverse Born series as
\be
\widetilde{\eta}=
\widetilde{\eta}_1+\widetilde{\eta}_2+\widetilde{\eta}_3+\widetilde{\eta}_4+\cdots,
\ee
where
\be
\widetilde{\eta}_j=\widetilde{\mathcal{K}}_j\phi^{\otimes j},\quad j\ge1.
\ee
For example, $\widetilde{\eta}_3$, $\widetilde{\eta}_4$ can be calculated as 
\ba
\widetilde{\eta}_3
&=
\widetilde{\mathcal{K}}_3\phi^{\otimes3}
=-\left(\widetilde{\mathcal{K}}_2K_2\otimes K_1\right)
\mathcal{K}_1\phi\otimes\mathcal{K}_1\phi\otimes\mathcal{K}_1\phi
\\
&=
\mathcal{K}_1K_2(\mathcal{K}_1K_2)\otimes (\mathcal{K}_1K_1)
(\mathcal{K}_1\phi\otimes\mathcal{K}_1\phi\otimes\mathcal{K}_1\phi),
\ea
\ba
\widetilde{\eta}_4
&=
\widetilde{\mathcal{K}}_4\phi^{\otimes4}
=-\left(\widetilde{\mathcal{K}}_3K_2\otimes K_1\otimes K_1\right)
\mathcal{K}_1\phi\otimes\mathcal{K}_1\phi\otimes\mathcal{K}_1\phi\otimes\mathcal{K}_1\phi
\\
&=
-\left((\mathcal{K}_1K_2(\mathcal{K}_1K_2)\otimes (\mathcal{K}_1K_1))
(\mathcal{K}_1K_2\otimes\mathcal{K}_1K_1\otimes\mathcal{K}_1K_1)
\right)
\\
&\times
\mathcal{K}_1\phi\otimes\mathcal{K}_1\phi\otimes\mathcal{K}_1\phi
\otimes\mathcal{K}_1\phi.
\ea

To consider the convergence of the reduced Born inverse series, we introduce
\be
\mu=\alpha_0\sup_{x\in\omega}\|G(x,\cdot)\|_{L^{\infty}(\omega)},\quad
\nu=\alpha_0|\omega|^{1/2}\sup_{x\in\omega}\|G(x,\cdot)\|_{L^p(\pp\Omega)}.
\ee
The following estimates are obtained for the forward operators $K_j$.

\begin{lem}[Moskow-Schotland \cite{Moskow-Schotland08}]
\be
\|K_j\|\le\nu\mu^{j-1},\quad j\ge1.
\ee
\label{MS08a}
\end{lem}

\begin{prop}
We assume that there exists $M\in(0,1)$ such that $\nu\|\mathcal{K}_1\|\le M$. Then the reduced inverse Born series converges if $\nu^{-1}\mu\|\phi\|_{L^p(\pp\Omega)}<1$.
\end{prop}

\begin{proof}
Using Lemma \ref{MS08a}, we begin with
\ba
\left\|\widetilde{\mathcal{K}}_n\right\|
&=
\left\|(\widetilde{\mathcal{K}}_{n-1}K_2\otimes K_1^{\otimes(n-2)})
\mathcal{K}_1^{\otimes n}\right\|
\\
&\le
\|\widetilde{\mathcal{K}}_{n-1}\|\|K_2\|\|K_1\|^{n-2}\|\mathcal{K}_1\|^n
\\
&\le
\|\widetilde{\mathcal{K}}_{n-1}\|\nu^{n-1}\mu\|\mathcal{K}_1\|^n
\\
&\le
\nu^{n(n-1)/2}\mu^{n-1}\|\mathcal{K}_1\|^{(n+1)n/2}
=
\frac{1}{\mu}\left(\frac{\mu}{\nu}\right)^n
\left(\nu\|\mathcal{K}_1\|\right)^{(n+1)n/2}.
\ea
We have
\ba
\|\eta\|_{L^{\infty}(\omega)}
&=
\sum_{n=1}^{\infty}\left\|\widetilde{\mathcal{K}}_n\phi^{\otimes n}\right\|_{L^{\infty}(\omega)}
\le
\sum_{n=1}^{\infty}\left\|\widetilde{\mathcal{K}}_n\right\|\|\phi\|_{L^p(\pp\Omega)}^n
\\
&\le
\frac{1}{\mu}\sum_{n=1}^{\infty}M^{(n+1)n/2}\left(\frac{\mu}{\nu}\|\phi\|_{L^p(\pp\Omega)}\right)^n.
\ea
Thus the proof is complete.
\end{proof}

From (\ref{iterativeRed1}), we have
\ba
\widetilde{\eta}^{(2)}&=\eta_1+\eta_2,
\\
\widetilde{\eta}^{(3)}&=
\eta_1+\eta_2-\mathcal{K}_1(K_2\eta_2\otimes\eta_1)
\simeq
\eta_1+\eta_2+\widetilde{\eta}_3
\\
\widetilde{\eta}^{(4)}&\simeq
\eta_1+\eta_2+\widetilde{\eta}_3-\mathcal{K}_1K_2\widetilde{\eta}_3\otimes\eta_1
\\
&=
\eta_1+\eta_2+\widetilde{\eta}_3-\mathcal{K}_1K_2\left(
\mathcal{K}_1K_2(\mathcal{K}_1K_2)\otimes(\mathcal{K}_1K_1)(\mathcal{K}_1\phi\otimes\mathcal{K}_1\phi\otimes\mathcal{K}_1\phi)\right)\otimes\eta_1
\\
&\simeq
\eta_1+\eta_2+\widetilde{\eta}_3+\widetilde{\eta}_4.
\ea
In this way, the iterative method (\ref{iterativeRed1}) corresponds to the reduced inverse Born series. Theorem \ref{relation_ite_rIBS} below states that the relation $\widetilde{\eta}^{(n)}\simeq\sum_{j=1}^n\widetilde{\eta}_j$ generally holds.

Recall
\be
\widetilde{\mathcal{K}}_2=-\mathcal{K}_1K_2\mathcal{K}_1^{\otimes2}.
\ee

\begin{prop}
For $n\ge3$,
\ba
\widetilde{\mathcal{K}}_n
&\simeq
(-1)^{n-1}\mathcal{K}_1K_2\left(\mathcal{K}_1K_2\otimes\mathcal{K}_1K_1\right)\left(\mathcal{K}_1K_2\otimes(\mathcal{K}_1K_1)^{\otimes2}\right)
\\
&\cdots
\left(\mathcal{K}_1K_2\otimes(\mathcal{K}_1K_1)^{\otimes(n-2)}\right)
\mathcal{K}_1^{\otimes n}.
\ea
\label{Kappa1K2otimesKappa1K1}
\end{prop}

For example,
\be
\widetilde{\mathcal{K}}_3\simeq
-\widetilde{\mathcal{K}}_2\left(K_2\otimes K_1\right)\mathcal{K}_1^{\otimes3}=
\mathcal{K}_1K_2\left(\mathcal{K}_1K_2\otimes\mathcal{K}_1K_1\right)\mathcal{K}_1^{\otimes3}.
\ee

\begin{proof}
Suppose
\ba
\widetilde{\mathcal{K}}_k
&\simeq
(-1)^{k-1}\mathcal{K}_1K_2\left(\mathcal{K}_1K_2\otimes\mathcal{K}_1K_1\right)
\left(\mathcal{K}_1K_2\otimes\mathcal{K}_1K_1\otimes\mathcal{K}_1K_1\right)
\\
&\cdots
\left(\mathcal{K}_1K_2\otimes(\mathcal{K}_1K_1)^{\otimes(k-2)}
\right)\mathcal{K}_1^{\otimes k}
\ea
for some $k\ge3$. In particular, the relation holds for $k=3$. Then we have
\ba
\widetilde{\mathcal{K}}_{k+1}
&=
-\widetilde{\mathcal{K}}_k\left(K_2\otimes K_1^{\otimes(k-1)}\right)\mathcal{K}_1^{\otimes(k+1)}
\\
&\simeq
(-1)^k\Biggl[\mathcal{K}_1K_2\left(\mathcal{K}_1K_2\otimes\mathcal{K}_1K_1\right)
\left(\mathcal{K}_1K_2\otimes\mathcal{K}_1K_1\otimes\mathcal{K}_1K_1\right)
\\
&\cdots
\left(\mathcal{K}_1K_2\otimes(\mathcal{K}_1K_1)^{\otimes(k-2)}
\right)\mathcal{K}_1^{\otimes k}\Biggr]
\left(K_2\otimes K_1^{\otimes(k-1)}\right)\mathcal{K}_1^{\otimes(k+1)}
\\
&=
(-1)^k\Biggl[\mathcal{K}_1K_2\left(\mathcal{K}_1K_2\otimes\mathcal{K}_1K_1\right)
\left(\mathcal{K}_1K_2\otimes\mathcal{K}_1K_1\otimes\mathcal{K}_1K_1\right)
\\
&\cdots
\left(\mathcal{K}_1K_2\otimes(\mathcal{K}_1K_1)^{\otimes(k-2)}\right)\Biggr]
\left(\mathcal{K}_1K_2\otimes(\mathcal{K}_1K_1)^{\otimes(k-1)}\right)
\mathcal{K}_1^{\otimes(k+1)}
\\
\\
&=
(-1)^k\mathcal{K}_1K_2\left(\mathcal{K}_1K_2\otimes\mathcal{K}_1K_1\right)
\left(\mathcal{K}_1K_2\otimes\mathcal{K}_1K_1\otimes\mathcal{K}_1K_1\right)
\\
&\cdots
\left(\mathcal{K}_1K_2\otimes(\mathcal{K}_1K_1)^{\otimes(k-1)}\right)
\mathcal{K}_1^{\otimes(k+1)}.
\ea
Thus the relation holds for $k+1$. The proof is complete by mathematical induction.
\end{proof}

\begin{thm}
\be
\widetilde{\eta}^{(n)}\simeq\sum_{j=1}^n\widetilde{\eta}_j,\quad n\ge1.
\ee
\label{relation_ite_rIBS}
\end{thm}

\begin{proof}
We have by definition
\be
\widetilde{\eta}^{(1)}=\mathcal{K}_1\phi=\eta_1.
\ee

Suppose for some $n\ge2$,
\be
\widetilde{\eta}^{(n-1)}\simeq\sum_{j=1}^{n-1}\widetilde{\eta}_j.
\ee
Using Proposition \ref{Kappa1K2otimesKappa1K1},
\ba
\widetilde{\eta}_n
&=
\widetilde{\mathcal{K}}_n\phi^n
\\
&\simeq
(-1)^{n-1}\mathcal{K}_1K_2\left(\mathcal{K}_1K_2\otimes\mathcal{K}_1K_1\right)\left(\mathcal{K}_1K_2\otimes\mathcal{K}_1K_1\otimes\mathcal{K}_1K_1\right)
\\
&\cdots
\left(\mathcal{K}_1K_2\otimes(\mathcal{K}_1K_1)^{\otimes(n-2)}\right)
\left(\mathcal{K}_1\phi\right)^{\otimes n}
\\
&\simeq
(-1)^{n-1}\mathcal{K}_1K_2\left(\mathcal{K}_1K_2\otimes\mathcal{K}_1K_1\right)\left(\mathcal{K}_1K_2\otimes\mathcal{K}_1K_1\otimes\mathcal{K}_1K_1\right)
\\
&\cdots
\left(\mathcal{K}_1K_2\otimes(\mathcal{K}_1K_1)^{\otimes(n-3)}\right)
\left[
\left(\mathcal{K}_1K_2\otimes(\mathcal{K}_1K_1)^{\otimes(n-3)}\right)
\left(\mathcal{K}_1\phi\right)^{\otimes (n-1)}\right]
\otimes\left(\mathcal{K}_1\phi\right)
\\
&\simeq
(-1)^{n-1}\mathcal{K}_1K_2\left(\mathcal{K}_1K_2\otimes\mathcal{K}_1K_1\right)\left(\mathcal{K}_1K_2\otimes\mathcal{K}_1K_1\otimes\mathcal{K}_1K_1\right)
\cdots
\left(\mathcal{K}_1K_2\otimes(\mathcal{K}_1K_1)^{\otimes(n-4)}\right)
\\
&\cdots
\left(\mathcal{K}_1K_2\otimes(\mathcal{K}_1K_1)^{\otimes(n-4)}\right)
\left[
\left(\mathcal{K}_1K_2\otimes(\mathcal{K}_1K_1)^{\otimes(n-3)}\right)
\left(\mathcal{K}_1\phi\right)^{\otimes (n-1)}\right]
\otimes\left(\mathcal{K}_1\phi\right)
\\
&\simeq
(-1)^{n-1}\mathcal{K}_1K_2\Biggl[
\mathcal{K}_1K_2\left(\mathcal{K}_1K_2\otimes\mathcal{K}_1K_1\right)\cdots
\left(\mathcal{K}_1K_2\otimes(\mathcal{K}_1K_1)^{\otimes(n-3)}\right)
\left(\mathcal{K}_1\phi\right)^{\otimes (n-1)}\Biggr]
\otimes\left(\mathcal{K}_1\phi\right)
\\
&\simeq
-\mathcal{K}_1K_2\widetilde{\eta}_{n-1}\otimes\left(\mathcal{K}_1\phi\right)
\\
&\simeq
-\mathcal{K}_1K_2\left(\eta_r^{(n-1)}-\eta_r^{(n-2)}\right)\otimes\eta_r^{(1)}.
\ea
Hence,
\be
\widetilde{\eta}_n\simeq\widetilde{\eta}^{(n)}-\widetilde{\eta}^{(n-1)}.
\ee
The proof is complete by mathematical induction.
\end{proof}

\section{Revisiting inverse operators}
\label{sec:reduce_2}

Hoskins and Schotland substituted (\ref{invBorn}) into (\ref{Born}), and found \cite{Hoskins-Schotland22}
\be
\phi=K_1\left(\mathcal{K}_1\phi+\mathcal{K}_2\phi^{\otimes2}+\cdots\right)
+K_2\left(\mathcal{K}_1\phi+\mathcal{K}_2\phi^{\otimes2}+\cdots\right)\otimes\left(\mathcal{K}_1\phi+\mathcal{K}_2\phi^{\otimes2}+\cdots\right)+\cdots,
\ee
By equating terms of the same order in $\phi$ on both sides, we obtain
\ba
&
K_1\mathcal{K}_1=I,
\\
&
K_1\mathcal{K}_2+K_2(\mathcal{K}_1\otimes\mathcal{K}_1)=0,
\ea
and in general,
\be
K_1\mathcal{K}_j+\sum_{m=2}^jK_m\sum_{i_1+\cdots+i_m=j}\mathcal{K}_{i_1}\otimes\cdots\otimes\mathcal{K}_{i_m}=0,
\quad j\ge2.
\ee
Thus we arrive at
\begin{equation}
\mathcal{K}_j\simeq
-\mathcal{K}_1\sum_{m=2}^jK_m\sum_{i_1+\cdots+i_m=j}\mathcal{K}_{i_1}\otimes\cdots\otimes\mathcal{K}_{i_m},
\quad j\ge2.
\label{arriveat_2}
\end{equation}
For example, we have (\ref{p6eq1}) and
\be
\mathcal{K}_3\simeq
-\mathcal{K}_1K_2\left(\mathcal{K}_1\otimes\mathcal{K}_2+\mathcal{K}_2\otimes\mathcal{K}_1\right)
-\mathcal{K}_1K_3\mathcal{K}_1^{\otimes3}.
\ee

The lemma below means that only one term on the right-hand side of (\ref{arriveat_2}) essentially contributes.

\begin{lem}
For $j\ge2$,
\begin{equation}
\mathcal{K}_j\simeq
-\mathcal{K}_1K_2\left(\mathcal{K}_{j-1}\otimes\mathcal{K}_1\right).
\label{reduction_2}
\end{equation}
\label{dominant_2}
\end{lem}

\begin{proof}
Similar to Lemma \ref{dominant}, the proof can be done with mathematical induction. When $j=2$, we have
\be
\mathcal{K}_2\simeq-\mathcal{K}_1K_2\mathcal{K}_1\otimes\mathcal{K}_1.
\ee

Suppose the relation (\ref{reduction_2}) holds for $j\ge2$. Then we have
\be
-\mathcal{K}_1K_2\sum_{i_1+i_2=j}\mathcal{K}_{i_1}\otimes\mathcal{K}_{i_2+1}-
\mathcal{K}_1\sum_{m=3}^{j+1}K_m\sum_{i_1+\cdots+i_m=j+1}\mathcal{K}_{i_1}\otimes\cdots\otimes\mathcal{K}_{i_m}
\simeq0.
\ee
Using Lemma \ref{generalKjn}, the above relation can be shown as follows:
\ba
&
-\mathcal{K}_1K_2\sum_{i_1+i_2=j}\mathcal{K}_{i_1}\otimes\mathcal{K}_{i_2+1}-
\mathcal{K}_1\sum_{m=3}^{j+1}K_m\sum_{i_1+\cdots+i_m=j+1}\mathcal{K}_{i_1}\otimes\cdots\otimes\mathcal{K}_{i_m}
\\
&\simeq
\mathcal{K}_1K_2\sum_{i_1+i_2=j}\mathcal{K}_{i_1}\otimes\mathcal{K}_1K_2(\mathcal{K}_{i_2}\otimes\mathcal{K}_1)-
\mathcal{K}_1\sum_{m=3}^{j+1}K_m\sum_{i_1+\cdots+i_m=j+1}\mathcal{K}_{i_1}\otimes\cdots\otimes\mathcal{K}_{i_m}
\\
&\simeq
\mathcal{K}_1K_3\sum_{i_1+i_2=j}\mathcal{K}_{i_1}\otimes\mathcal{K}_{i_2}\otimes\mathcal{K}_1-
\mathcal{K}_1\sum_{m=3}^{j+1}K_m\sum_{i_1+\cdots+i_m=j+1}\mathcal{K}_{i_1}\otimes\cdots\otimes\mathcal{K}_{i_m}
\\
&\simeq
-\mathcal{K}_1K_3\sum_{i_1+i_2+i_3=j}\mathcal{K}_{i_1}\otimes\mathcal{K}_{i_2}\otimes\mathcal{K}_{i_3+1}
-\mathcal{K}_1\sum_{m=4}^{j+1}K_m\sum_{i_1+\cdots+i_m=j+1}\mathcal{K}_{i_1}\otimes\cdots\otimes\mathcal{K}_{i_m}
\\
&\simeq
-\mathcal{K}_1K_j\mathcal{K}_1\otimes\cdots\otimes\mathcal{K}_1\otimes\mathcal{K}_2
-\mathcal{K}_1K_{j+1}\mathcal{K}_1^{\otimes(j+1)}
\\
&\simeq
\mathcal{K}_1K_j\mathcal{K}_1\otimes\cdots\otimes\mathcal{K}_1\otimes\mathcal{K}_1K_2(\mathcal{K}_1\otimes\mathcal{K}_1)
-\mathcal{K}_1K_{j+1}\mathcal{K}_1^{\otimes(j+1)}
\\
&\simeq0.
\ea
Let us write (\ref{arriveat_2}) for $j+1$:
\ba
\mathcal{K}_{j+1}
&\simeq
-\mathcal{K}_1\sum_{m=2}^{j+1}K_m\sum_{i_1+\cdots+i_m=j+1}\mathcal{K}_{i_1}\otimes\cdots\otimes\mathcal{K}_{i_m}
\\
&\simeq
-\mathcal{K}_1K_2\sum_{i_1+i_2=j+1}\mathcal{K}_{i_1}\otimes\mathcal{K}_{i_2}
-\mathcal{K}_1\sum_{m=3}^{j+1}K_m\sum_{i_1+\cdots+i_m=j+1}\mathcal{K}_{i_1}\otimes\cdots\otimes\mathcal{K}_{i_m}
\\
&\simeq
-\mathcal{K}_1K_2\mathcal{K}_j\otimes\mathcal{K}_1
-\mathcal{K}_1K_2\sum_{i_1+i_2=j}\mathcal{K}_{i_1}\otimes\mathcal{K}_{i_2+1}
-\mathcal{K}_1\sum_{m=3}^{j+1}K_m\sum_{i_1+\cdots+i_m=j+1}\mathcal{K}_{i_1}\otimes\cdots\otimes\mathcal{K}_{i_m}
\\
&\simeq
-\mathcal{K}_1K_2\mathcal{K}_j\otimes\mathcal{K}_1.
\ea
Thus the proof is complete.
\end{proof}

Let us recursively define
\begin{equation}
\widehat{\mathcal{K}}_j
=-\mathcal{K}_1K_2\left(\widehat{\mathcal{K}}_{j-1}\otimes\widehat{\mathcal{K}}_1\right),\quad j\ge3,
\label{iteHS1}
\end{equation}
where $\widehat{\mathcal{K}}_1=\mathcal{K}_1$, $\widehat{\mathcal{K}}_2=\mathcal{K}_2$. Thus we can introduce another reduced inverse Born series as
\begin{equation}
\widehat{\eta}=
\widehat{\eta}_1+\widehat{\eta}_2+\widehat{\eta}_3+\widehat{\eta}_4+\cdots,
\label{rIBS2}
\end{equation}
where
\be
\widehat{\eta}_j=\widehat{\mathcal{K}}_j\phi^{\otimes j},\quad j\ge1.
\ee
For example, $\widehat{\eta}_1=\eta_1$, $\widehat{\eta}_2=\eta_2$,
\ba
\widehat{\eta}_3
&=
-\mathcal{K}_1K_2\left(\widehat{\mathcal{K}}_2\phi^{\otimes2}\otimes\mathcal{K}_1\phi\right)
=
\mathcal{K}_1K_2\left(\mathcal{K}_1K_2(\eta_1\otimes\eta_1)\otimes\eta_1\right),
\\
\widehat{\eta}_4
&=
-\mathcal{K}_1K_2\left(\widehat{\mathcal{K}}_3\phi^{\otimes3}\otimes\mathcal{K}_1\phi\right)
=
-\mathcal{K}_1K_2\left[
\mathcal{K}_1K_2\left(\mathcal{K}_1K_2(\eta_1\otimes\eta_1)\otimes\eta_1\right)\otimes\eta_1
\right].
\ea

From (\ref{iteHS1}), we have
\ba
\widetilde{\eta}^{(2)}&=\widehat{\eta}_1+\widehat{\eta}_2,
\\
\widetilde{\eta}^{(3)}&=
\widehat{\eta}_1+\widehat{\eta}_2-\mathcal{K}_1(K_2\widehat{\eta}_2\otimes\widehat{\eta}_1)
=
\widehat{\eta}_1+\widehat{\eta}_2+\widehat{\eta}_3
\\
\widetilde{\eta}^{(4)}&=
\widehat{\eta}_1+\widehat{\eta}_2+\widehat{\eta}_3-\mathcal{K}_1K_2\widehat{\eta}_3\otimes\eta_1
\\
&=
\widehat{\eta}_1+\widehat{\eta}_2+\widehat{\eta}_3+\widehat{\eta}_4.
\ea
Thus, the iterative method (\ref{iterativeRed1}) is equivalent to this reduced inverse Born series (\ref{rIBS2}).

\section{Numerical tests}
\label{sec:num}

\subsection{A simple invertible case}

If $K_1$ has an inverse and $\mathcal{K}_1K_1=I$, terms in the fast iterative scheme (\ref{iterativeRed1}) and in the inverse Born series (\ref{invBorn}) coincide. To demonstrate this, we consider a simple series:
\be
y=\sum_{n=1}^{\infty}y_n,\quad y\in\Rm^2,
\ee
where $y_1=Ax$ with $x\in\Rm^2$. Here $y_n$'s are recursively given by
\be
\{y_n\}_i=-\sum_{j=1}^2A_{ij}x_j\{y_{n-1}\}_j,\quad i=1,2.
\ee
Let us set
\be
A=\begin{pmatrix}0.1&0.2\\0.3&0.4\end{pmatrix},\quad
x=\begin{pmatrix}0.07\\0.08\end{pmatrix}.
\ee
We note that the matrix $A$ is invertible:
\be
A^{-1}=\begin{pmatrix}-20&10\\15&-5\end{pmatrix}.
\ee

With five significant digits, we have
\be
y=y_1+y_2+y_3+y_4+y_5=\begin{pmatrix}0.022031\\ 0.050908\end{pmatrix}.
\ee
This $y$ can be regarded as observation data.

According to (\ref{iterativeRed1}), we have $x^{(0)}=(0,0)^T$, $x^{(1)}=A^{-1}y$,
\be
x^{(n+1)}=x^{(n)}-A^{-1}K_2\left(x^{(n)}-x^{(n-1)}\right)\otimes x^{(1)},\quad n\ge1.
\ee
Here,
\be
\left\{K_2\left(x^{(n)}-x^{(n-1)}\right)\otimes x^{(1)}\right\}_i
=-\sum_{j=1}^2A_{ij}\left(x_j^{(n)}-x_j^{(n-1)}\right)\sum_{l=1}^2A_{jl}x_l^{(1)},\quad i\ge2,\quad n\ge1.
\ee
We obtain
\ba
x^{(1)}&=(0.0684578, 0.0759273)^T,
\\
x^{(2)}&=(0.0699660, 0.0797927)^T,
\\
x^{(3)}&=(0.0699993, 0.0799895)^T,
\\
x^{(4)}&=(0.0700000, 0.0799995)^T,
\\
x^{(5)}&=(0.0700000, 0.0800000)^T.
\ea

According to (5.2), we have $x_1=A^{-1}y$ and $x_j=\mathcal{K}_jy^{\otimes j}$ ($j\ge2$), where
\be
\mathcal{K}_jy^{\otimes j}=-\left(\sum_{m=1}^{j-1}\mathcal{K}_m\sum_{i_1+\cdots+i_m=j}K_{i_1}\otimes\cdots\otimes K_{i_m}\right)x_1^{\otimes j},
\quad j\ge2.
\ee
Here, $K_1x=Ax$,
\be
K_n\xi^{\otimes n}=-\sum_{l=1}^2A_{nl}\xi_l\left\{K_{n-1}\xi^{\otimes(n-1)}\right\}_l,\quad
\xi\in\Rm^2\quad n\ge2.
\ee
We obtain
\ba
x_1&=(0.0684578, 0.0759273)^T,
\\
x_1+x_2&=(0.0699660, 0.0797927)^T,
\\
x_1+x_2+x_3&=(0.0699993, 0.0799895)^T,
\\
x_1+x_2+x_3+x_4&=(0.0700000, 0.0799995)^T,
\\
x_1+x_2+x_3+x_4+x_5&=(0.0700000, 0.0800000)^T.
\ea
Thus, the inverse Born series (\ref{invBorn}) can be computed with the iteration (\ref{iterativeRed1}).

\subsection{Radial problem: Setup}

We consider the two-dimensional radial problem with spatially-oscillating illumination \cite{Machida23}. Let $\Omega$ be a disk of radius $R$ centered at the origin. We write $\eta=\eta(r)$, where $r$ is the radial coordinate. Let us set
\be
\eta(r)=\left\{\begin{aligned}
\eta_a,&\quad 0\le r\le a,
\\
0,&\quad a<r< R.
\end{aligned}\right.
\ee
We can write $x=(r,\theta)$, $\theta\in\Sm^1$. Assuming a spatially-oscillating source, the diffusion equation is given by
\begin{equation}
\begin{aligned}
-\Delta u(x)+k^2\left(1+\eta(r)\right)u(x)=\frac{e^{im\theta}}{r}\delta(r-R),&\quad x\in\Omega,
\\
\pp_{\nu}u+\beta u=0,&\quad x\in\pp\Omega
\end{aligned}
\label{num:diffeq}
\end{equation}
for $m=1,\dots,M_S$. We put $\alpha_0=k^2$ for later convenience.

Let us consider the Green's function which satisfies
\begin{equation}
\left\{\begin{aligned}
-\Delta G(x,x')+k^2G(x,x')=\delta(x-x'),&\quad x\in\Omega,
\\
\pp_{\nu}G+\beta G=0,&\quad x\in\pp\Omega.
\end{aligned}\right.
\label{num:Gfunc}
\end{equation}
We obtain
\be
G(x,x')=\frac{1}{2\pi}\sum_{m=-\infty}^{\infty}e^{im(\theta-\theta')}g_m(r,r'),
\ee
where $x'=(r',\theta')$, $\theta'\in\Sm^1$. Here, $g_m(r,r')$ satisfies
\ba
r^2\pp_r^2g_m(r,r')+r\pp_rg_m(r,r')-\left(k^2r^2+m^2\right)g_m(r,r')
&=
-r\delta(r-r'),
\\
\pp_rg_m(R,r')+\beta g_m(R,r')=0.
\ea
Let $I_m,K_m$ be the modified Bessel functions of the first and second kinds, respectively. Furthermore, $I'_m,K'_m$ are derivatives of $I_m,K_m$. We obtain
\ba
g_m(r,r')
&=
K_m\left(k\max(r,r')\right)I_m\left(k\min(r,r')\right)
\\
&-
\frac{\beta K_m(kR)+k K'_m(kR)}{\beta I_m(kR)+k I'_m(kR)}I_m(kr)I_m(kr').
\ea
We note that the relation $g_{-m}(r,r')=g_m(r,r')$. If $x\in\overline{\Omega}$, $x'\in\pp\Omega$, the function $g_m$ can be expressed as
\be
g_m(r,R)=K_m(kR)I_m(kr)-d_mI_m(kr),
\ee
where
\be
d_m=\frac{\beta K_m(kR)+k K'_m(kR)}{\beta I_m(kR)+k I'_m(kR)}I_m(kR),
\ee
or
\ba
g_m(r,R)
&=
\frac{kK_m(kR)I'_m(kR)-kK'_m(kR)I_m(kR)}{\beta I_m(kR)+kI'_m(kR)}I_m(kr)
\\
&=
\frac{I_m(kr)}{R\left(\beta I_m(kR)+k I'_m(kR)\right)}
=:\widetilde{g}_m(r).
\ea

\subsection{Radial problem: Forward data}

Define
\be
\omega=\left\{x;\;r_x\le a\right\}.
\ee
Let us consider the Green's function $G_a$ which corresponds to the diffusion equation (\ref{num:diffeq}). We can write $G_a=v$ in $\omega$ and $G_a=w$ in $\Omega$. Here,
\ba
-\Delta v+\left(1+\eta_a\right)k^2v=\delta(x-x'),&\quad x\in\omega,
\\
-\Delta w+k^2w=0,&\quad x\in\Omega\setminus\omega.
\ea
On the interface and boundary, we have
\begin{equation}\left\{\begin{aligned}
v=w,&\quad x\in\pp\omega,
\\
\pp_{\nu}v=\pp_{\nu}w,&\quad x\in\pp\omega,
\\
\pp_{\nu}w+\beta w=0,&\quad x\in\pp\Omega.
\end{aligned}\right.
\label{intbd}
\end{equation}

The solutions $v,w$ can be written as
\be
v(x)=\frac{1}{2\pi}\sum_{m=-\infty}^{\infty}a_me^{im(\theta-\theta')}
I_m\left(\sqrt{1+\eta_a}kr\right),\quad x\in\omega,
\ee
\be
w(x)=
G_0(x,x')+\frac{1}{2\pi}\sum_{m=-\infty}^{\infty}e^{im(\theta-\theta')}
\left(b_mK_m(kr)+c_mI_m(kr)\right),\quad x\in\Omega\setminus\omega,
\ee
and $a_m,b_m,c_m$ will be determined using (\ref{intbd}). Here, $G_0(x,x')$ satisfies
\be
-\Delta G_0(x,x')+k^2G_0(x,x')=\delta(x-x'),\quad x\in\Rm^2,
\ee
where $G_0\to0$ as $|x|\to\infty$. We obtain
\be
G_0(x,x')=\frac{1}{2\pi}\sum_{m=-\infty}^{\infty}e^{im(\theta-\theta')}I_m\left(kr\right)K_m(kR).
\ee
The coefficients $a_m,b_m,c_m$ are solutions to the following linear system \cite{Moskow-Schotland09}:
\ba
&
\begin{pmatrix}
I_m(\sqrt{1+\eta_a}ka) & -K_m(ka) & -I_m(ka) \\
\sqrt{1+\eta_a}kI'_m(\sqrt{1+\eta_a}ka) & -kK'_m(ka) & -kI'_m(ka) \\
0 & \beta K_m(kR)+kK'_m(kR) & \beta I_m(kR)+kI'_m(kR)
\end{pmatrix}
\begin{pmatrix}
a_m \\ b_m \\ c_m
\end{pmatrix}
\\
&=
\begin{pmatrix}
I_m(ka)K_m(kR) \\ kI'_m(ka)K_m(kR) \\ kI_m(kR)K'_m(kR)+\beta I_m(kR)K_m(kR)
\end{pmatrix}.
\ea
We obtain
\ba
u(x)
&=
\int_0^{2\pi}\int_0^RG_a(x,x')e^{im\theta'}\,dr'd\theta'
\\
&=Re^{im\theta}\left(I_m(kR)K_m(kR)+b_mK_m(kR)+c_mI_m(kR)\right).
\ea
We have
\be
u_0(x)=\int_0^{2\pi}\int_0^RG(x,x')\frac{e^{im\theta'}}{r'}\delta(r'-R)r'\,dr'd\theta'
=e^{im\theta}g_m(r,R)=e^{im\theta}\widetilde{g}_m(r).
\ee

We observe $u,u_0$ at $r=R$, $\theta=0$ (i.e., $x=(R,0)$) for the forward data:
\be
\phi(m)=g_m(R,R)-R\left(I_m(kR)K_m(kR)+b_mK_m(kR)+c_mI_m(kR)\right).
\ee
%To use the above $\phi(m)$ as observation data, the Gaussian noise with the standard deviation of $3\%$ of the standard deviation of $u_0-u$ was added to each of $u,u_0$.

\subsection{Radial problem: Reconstruction}

Let us begin with the following integral equation:
\be
u_0(x)-u(x)=k^2\int_{\Omega}G(x,x')\eta(r')u(x')\,dx',\quad x\in\overline{\Omega}.
\ee
We have
\ba
u_0(x)-u(x)
&=
k^2\int_{\Omega}G(x,x')\eta(r')u_0(x')\,dx'
\\
&-
k^4\int_{\Omega}\int_{\Omega}G(x,x')\eta(r')G(x',x'')\eta(r'')u_0(x'')\,dx'dx''+\cdots,
\ea
where $x''=(r'',\theta'')$. We can proceed as
\ba
&
u_0(x)-u(x)=
\frac{k^2}{(2\pi)^2}\sum_{m_1=-\infty}^{\infty}\sum_{m_2=-\infty}^{\infty}
\int_{\Omega}\int_{\Omega}e^{im_1(\theta-\theta')}e^{im_2(\theta'-\theta_0)}
\\
&\qquad\times
g_{m_1}(r,r')\eta(r')g_{m_2}(r',R)\frac{e^{im\theta_0}}{r_0}\,dx_0dx'
\\
&\qquad-
\frac{k^4}{(2\pi)^3}\sum_{m_1=-\infty}^{\infty}\sum_{m_2=-\infty}^{\infty}\sum_{m_3=-\infty}^{\infty}\int_{\Omega}\int_{\Omega}\int_{\Omega}
e^{im_1(\theta-\theta')}e^{im_2(\theta'-\theta'')}e^{im_3(\theta''-\theta_0)}
\\
&\qquad\times
g_{m_1}(r,r')\eta(r')g_{m_2}(r',r'')\eta(r'')g_{m_3}(r'',R)
\frac{e^{im\theta_0}}{r_0}\,dx_0dx'dx''+\cdots
\\
&=
k^2Re^{im\theta}\int_0^Rg_m(r,r')\eta(r')g_m(r',R)r'\,dr'
\\
&-
k^4Re^{im\theta}\int_0^R\int_0^R
g_m(r,r')\eta(r')g_m(r',r'')\eta(r'')g_m(r'',R)r'r''\,dr'dr''+\cdots.
\ea

By setting $r=R$, $\theta=0$, we obtain
\ba
\phi(m)
&=
k^2R\int_0^Rg_m(R,r')\eta(r')g_m(r',R)r'\,dr'
\\
&-
k^2\int_0^Rg_m(R,r')\eta(r')\left(k^2R\int_0^R
g_m(r',r'')\eta(r'')g_m(r'',R)r''\,dr''\right)r'\,dr'+\cdots
\\
&=
K_1\eta+K_2\eta\otimes\eta+\cdots.
\ea
Here,
\ba
(K_1\eta)(m)
&=
k^2R\int_0^Rg_m(r',R)r'g_m(r',R)\eta(r')\,dr',
\\
\left(K_j\eta^{\otimes j}\right)(m)
&=
(-1)^{j-1}k^{2j}R\int_0^R\cdots\int_0^Rg_m(r_1,R)\eta(r_1)g_m(r_1,r_2)\cdots
g_m(r_{j-1},r_j)
\\
&\times
\eta(r_j)g(r_j,R)r_1\cdots r_j\,dr_1\dots dr_j,\quad
j\ge2.
\ea
In particular,
\ba
\left(K_2\eta\otimes\eta\right)(m)
&=
-k^4R\int_0^R\int_0^Rg_m(r',R)\eta(r')g_m(r',r'')\eta(r'')g_m(r'',R)r'r''\,dr'dr''
\\
&=
k^2\int_0^Rg_m(r',R)\eta(r')u_1(r')r'\,dr',
\ea
where
\be
u_1(r)=-k^2R\int_0^Rg_m(r,r')r'g_m(r',R)\eta(r')\,dr'.
\ee

We note that
\be
\left(K_2\widetilde{\eta}\otimes\eta_1\right)(m)=
-k^4R\int_0^R\int_0^Rg_m(r',R)\widetilde{\eta}(r')g_m(r',r'')\eta_1(r'')g_m(r'',R)r'r''\,dr'dr''.
\ee
For numerical calculation, $r$ is discretized using the step size $\Delta r$ and $r_j$ is the $j$th position ($j=1,\dots,N_r$) of $r$. Let $\widetilde{\bv{\eta}}$ be a vector whose $j$th component is $\widetilde{\eta}(r_j)$. We use a matrix $\underline{\mathcal{K}}_1$ which corresponds to $\mathcal{K}_1$. In this radial problem, $\underline{\mathcal{K}}_1$ is an $N_r\times M_S$ matrix. After discritization, we can write $-\mathcal{K}_1(K_2\widetilde{\eta}\otimes\eta_1)$ in matrix-vector form as $B\widetilde{\bv{\eta}}$, where $B$ is an $N_r\times N_r$ square matrix:
\be
B_{ij}=k^4R(\Delta r)^d\sum_{m=1}^{M_S}\left\{\underline{\mathcal{K}}_1\right\}_{im}
g_m(r'_j,R)r'_j\int_0^Rg_m(r'_j,r'')\eta_1(r'')g_m(r'',R)r''\,dr'',
\ee
where $1\le i\le N_r$, $1\le j\le N_r$. Thus, the fast iteration can be written as
\begin{equation}
\widetilde{\bv{\eta}}^{(n+1)}=
\widetilde{\bv{\eta}}^{(n)}+B\left(\widetilde{\bv{\eta}}^{(n)}-\widetilde{\bv{\eta}}^{(n-1)}\right),
\quad n\ge1,
\label{iterativeRed2}
\end{equation}
with $\bv{\eta}^{(0)}=\bv{0}$, $\bv{\eta}^{(1)}=\underline{\mathcal{K}}_1\bv{\phi}$. To obtain $\underline{\mathcal{K}}_1$, we compute $K_1^+$ by singular value decomposition using the $M_S\times N_r$ matrix which corresponds to $K_1$, and calculate $\underline{\mathcal{K}}_1$ by discarding small singular values.

We set $M_S=N_r=90$ and used $23$ largest singular values. Furthermore, we set
\be
k=1,\quad R=3,\quad a=1.5,\quad \eta_a\equiv0.2,\quad \beta=3
\ee

Reconstructed $\eta$ for $\eta_a=0.2$ is shown in Fig.~\ref{fig1_02}. Using the fast iterative scheme in (\ref{iterativeRed1}), $\eta$ is reconstructed in the left panel of Fig.~\ref{fig1_02}. In Fig.~\ref{fig1_02}, the best reconstruction which can be achieved is $\eta_{\rm proj}$ in (\ref{etaproj}). To compare, the reconstructed $\eta$ by the inverse Born series (\ref{invBorn}) is shown in the right panel of Fig.~\ref{fig1_02}. In Fig.~\ref{fig1_04}, we performed reconstruction for $\eta_a=0.4$.

The computation time for Fig.~\ref{fig1_02} with the iteration (\ref{iterativeRed1}) was $1.7\,{\rm min}$, whereas the calculation for the inverse Born series (\ref{invBorn}) took $141.1\,{\rm min}$ (Mathematica on MacBook Pro). The difference in computation time grows exponentially if higher-order terms are added. Although the computation times are different, we can see that reconstructed curves are almost indistinguishable in Fig.~\ref{fig1_02}. Figure \ref{fig2}(Left) compares numerical results the fast iterative scheme in (\ref{iterativeRed1}) and the inverse Born series (\ref{invBorn}) for $\eta_a=0.2$. Quite similar reconstructions are obtained in Fig.~\ref{fig1_04} both for the fast iteration (\ref{iterativeRed1}) and the inverse Born series (\ref{invBorn}). In Fig.~\ref{fig2}(Right), the 5th-order approximations for the fast iteration (\ref{iterativeRed1}) and the inverse Born series (\ref{invBorn}) are slightly different.

\begin{figure}[ht]
\centering
\includegraphics[width=0.45\textwidth]{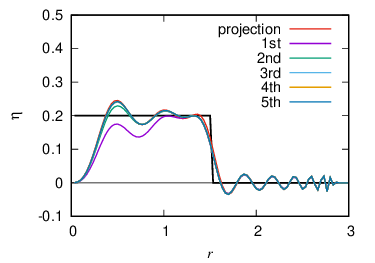}
\hspace{5mm}
\includegraphics[width=0.45\textwidth]{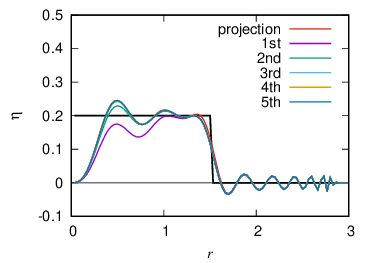}
\caption{
In the case of $\eta_a\equiv0.2$. (Left) The reconstructed $\eta$ by the fast iterative scheme (\ref{iterativeRed1}). (Right) The reconstructed $\eta$ by the inverse Born series (\ref{invBorn}). In both panels, the black line shows the true shape of the target and the red line is the best reconstruction given in (\ref{etaproj}); furthermore, the purple, green, light blue, ocher, and dark blue lines show 1st (linear), 2nd, 3rd, 4th, and 5th reconstructions, respectively.
}
\label{fig1_02}
\end{figure}

\begin{figure}[ht]
\centering
\includegraphics[width=0.45\textwidth]{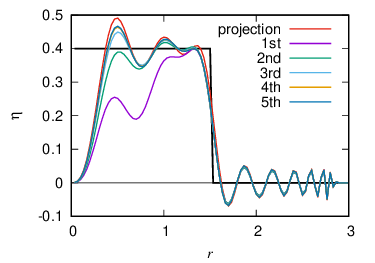}
\hspace{5mm}
\includegraphics[width=0.45\textwidth]{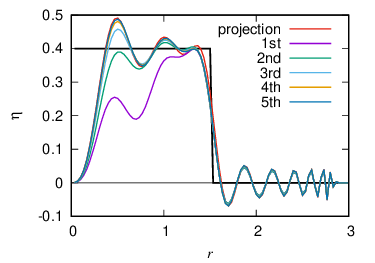}
\caption{
Same as Fig.~\ref{fig1_02} but $\eta_a\equiv0.4$.
}
\label{fig1_04}
\end{figure}

\begin{figure}[ht]
\centering
\includegraphics[width=0.45\textwidth]{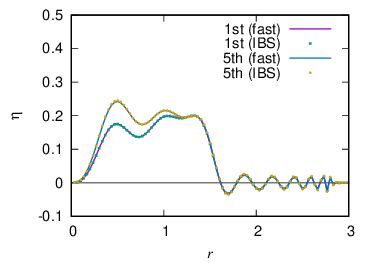}
\hspace{5mm}
\includegraphics[width=0.45\textwidth]{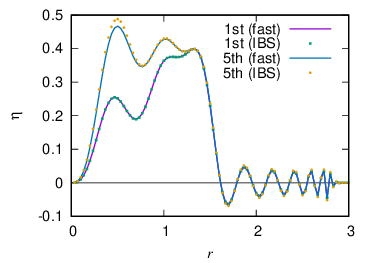}
\caption{
Comparison of the reconstructed $\eta$ by the fast iterative scheme (\ref{iterativeRed1}) and the inverse Born series (\ref{invBorn}). Results for (Left) $\eta_a=0.2$ and (Right) $\eta_a=0.4$ are shown. The purple and dark blue lines show reconstructions by the 1st- and 5th-order fast iterations, and green and ocher dots show reconstructions by the 1st- and 5th-order inverse Born series.
}
\label{fig2}
\end{figure}

\section{Concluding remarks}
\label{concl}

When the inverse Born series is derived, it is assumed that $\mathcal{K}_1K_1$ is close to the identity. This condition is achieved if regularization is done properly. In this paper, we showed that different expressions of a solution is possible under the condition. In particular in the proposed fast iterative scheme, all we need are $\mathcal{K}_1$ and $K_2$ to invert the Born series.

%The quality of the approximation depends on how $\mathcal{K}_1K_1$ is close to $I$. When the inverse problem is severely ill-posed such as the inverse problem of determining $\eta$ in the time-independent diffusion equation (\ref{diffeq1}), there is a noticeable difference between $\mathcal{K}_1K_1$ and $I$. But in this case, a good reconstruction can not be expected due to the severe ill-posedness. On the other hand, when the inverse problem is not so ill-posed and $\mathcal{K}_1K_1$ is almost the identity, the proposed fast iterative scheme (\ref{iterativeRed1}) provides a good reconstruction.

%\section*{Acknowledgments}

%\newpage
%\setcounter{section}{1}
%\appendix

\end{document}